\documentclass[12pt,twoside]{article}
\usepackage[left=1.5cm,right=1.5cm,top=3cm,bottom=2cm]{geometry}
\usepackage{paralist}
\usepackage{marginnote}
\usepackage{amsfonts}
\usepackage{enumitem}
\usepackage{tikz}
\usetikzlibrary{matrix}
\usepackage{amsthm}
\usepackage{amsrefs}
\usepackage{amssymb}
\usepackage{amsmath}
\usepackage{graphicx}
\usepackage{varwidth}
\usepackage{hyperref}
\usepackage{fancyhdr}
\usepackage{stmaryrd}
\usepackage{multirow}

\usepackage{blkarray}
\usepackage[autostyle]{csquotes}
\usepackage[mathscr]{euscript}
\hypersetup{
    colorlinks=true,
    linkcolor=black,
    filecolor=magenta,      
    urlcolor=black,
    citecolor=black
}
\DeclareMathAlphabet{\mathcalligra}{T1}{calligra}{c}{h}
\providecommand{\U}[1]{\protect\rule{.1in}{.1in}}
\newtheorem{theorem}{Theorem}[section]
\newtheorem{proposition}[theorem]{Proposition}
\newtheorem{lemma}[theorem]{Lemma}

\newtheorem{corollary}[theorem]{Corollary}

\newtheorem{remark}[theorem]{Remark}
\let\oldremark\remark
\renewcommand{\remark}{\oldremark\normalfont}

\newtheorem{example}[theorem]{Example}
\let\oldexample\example
\renewcommand{\example}{\oldexample\normalfont}

\newtheorem{examples}[theorem]{Examples}
\let\oldexamples\examples
\renewcommand{\examples}{\oldexamples\normalfont}

\newcommand{\R}{\mathbb{R}}

\def\<{{\langle}}
\def\>{{\rangle}}

\def\bea{\begin{eqnarray*} }
\def\eea{\end{eqnarray*} }
\def\be{\begin{equation} }
\def\ee{\end{equation} }

\def\qed{\ifhmode\unskip\nobreak\fi\ifmmode\ifinner
\else\hskip5 pt \fi\fi\hbox{\hskip5 pt \vrule width4 pt
height6 pt  depth1.5 pt \hskip 1pt }}

\DeclareMathOperator*{\supp}{supp}

\DeclareMathOperator*{\grad}{grad}
\DeclareMathOperator*{\ess}{ess}
\DeclareMathOperator*{\Lip}{Lip}

\DeclareMathOperator*{\dv}{dv}
\begin{document}

\title{Normal covering spaces with maximal bottom of spectrum}
\author{Panagiotis Polymerakis}
\date{}

\maketitle

\renewcommand{\thefootnote}{\fnsymbol{footnote}}
\footnotetext{\emph{Date:} \today} 
\renewcommand{\thefootnote}{\arabic{footnote}}

\renewcommand{\thefootnote}{\fnsymbol{footnote}}
\renewcommand{\thefootnote}{\arabic{footnote}}

\renewcommand{\thefootnote}{\fnsymbol{footnote}}
\renewcommand{\thefootnote}{\arabic{footnote}}

\begin{abstract}
We study the property of spectral-tightness of Riemannian manifolds, which means that the bottom of the spectrum of the Laplacian separates the universal covering space from any other normal covering space of a Riemannian manifold. We prove that spectral-tightness of a closed Riemannian manifold is a topological property characterized by its fundamental group. As an application, we show that a non-positively curved, closed Riemannian manifold is spectrally-tight if and only if the dimension of its Euclidean local de Rham factor is zero. In their general form, our results extend the state of the art results on the bottom of the spectrum under Riemannian coverings.
\end{abstract}

\section{Introduction}

The spectrum of the Laplacian on a Riemannian manifold is an interesting isometric invariant, which attracted much attention over the last years.
Aiming to a better comprehension of its relations with the geometry of the underlying manifold, its behavior under maps between Riemannian manifolds that respect the geometry of the manifolds to some extent has been studied.
In particular, there are various results on the behavior of the spectrum under Riemannian coverings and open questions arising from them.

To be more precise, let $p \colon M_{2} \to M_{1}$ be a Riemannian covering. Then the bottoms of the spectra of the Laplacians satisfy $\lambda_{0}(M_{1}) \leq \lambda_{0}(M_{2})$. Brooks was the first one to investigate when the equality holds, and this is closely related to the notion of amenability. The covering $p$ is called amenable if the monodromy action of $\pi_{1}(M_{1})$ on the fiber of $p$ is amenable. It is noteworthy that a normal covering is amenable if and only if its deck transformation group is amenable. Brooks proved in \cite{Brooks} that the universal covering of closed (that is, compact and without boundary) manifold preserves the bottom of the spectrum if and only if it is amenable. This result has been generalized in various ways over the last years (cf. for instance, the survey \cite{surv}). In \cite{BMP1}, we showed that if $p$ is amenable, then $\lambda_{0}(M_{1}) = \lambda_{0}(M_{2})$, without imposing any assumptions on the geometry or the topology of the manifolds. The converse implication is not true in general, but holds under a natural condition involving the bottom $\lambda_{0}^{\ess}(M_{1})$ of the essential spectrum of the Laplacian. More specifically, according to \cite[Theorem 1.2]{Mine2}, if $\lambda_{0}(M_{2}) =\lambda_{0}(M_{1}) < \lambda_{0}^{\ess}(M_{1})$, then $p$ is amenable. The situation is very unclear in the case where $\lambda_{0}(M_{1}) = \lambda_{0}^{\ess}(M_{1})$, as pointed out in \cite[Question 1.5]{surv}.

Given a normal Riemannian covering $p \colon M_{2} \to M_{1}$, besides the aforementioned inequality, we have that $\lambda_{0}(M_{2}) \leq \lambda_{0}(\tilde{M})$, where $\tilde{M}$ is the universal covering space of $M_{1}$. One of the purposes of this paper is to examine the validity of the equality. This fits into the study of manifolds with maximal bottom of spectrum under some constraint.
Examples of remarkable works on this problem are \cite{LW2} and \cite{JLW}, which focus on complete manifolds with Ricci curvature bounded from below and quotients of symmetric spaces of non-compact type, respectively.
This setting, where $M_{2}$ is a normal covering space of $M_{1}$, may seem more restrictive, but our goal is actually different. 
In addition to obtaining information about the maximizer $M_{2}$, we want to characterize the existence of a maximizer (different from  the universal covering space) in terms of properties of $M_{1}$.

Our motivation is to investigate to what extent the bottom of the spectrum of the Laplacian separates the universal covering space from the rest of the normal covering spaces of a Riemannian manifold. The corresponding question about the exponential volume growth has been addressed in \cite{MR2390354}. To be more precise, if $M_{2}$ is a normal covering space of $M_{1}$, then the exponential volume growths satisfy $\mu(M_{1}) \leq \mu(M_{2})$. Hence, the universal covering space $\tilde{M}$ of $M_{1}$ has maximal bottom of spectrum and maximal exponential volume growth among all the normal covering spaces of $M_{1}$. A Riemannian manifold $M_{1}$ is called \textit{spectrally-tight} or \textit{growth tight} if the universal covering space of $M_{1}$ is the unique normal covering space of $M_{1}$ with maximal bottom of spectrum or maximal exponential volume growth, respectively. Sambusetti proved in \cite{MR2390354} that negatively curved, closed manifolds are growth tight, which yields that negatively curved, closed, locally symmetric spaces are spectrally-tight.
One of the aims of this paper is to study the notion of spectral-tightness, establish that negatively curved, closed manifolds enjoy this property, and more generally, characterize this property for non-positively curved, closed Riemannian manifolds.

To set the stage, we consider Riemannian coverings $p \colon M_{2} \to M_{1}$ and $q \colon M_{1} \to M_{0}$, with $q$ normal, a Schr\"{o}dinger operator $S_{0}$ on $M_{0}$ and its lifts $S_{1}$, $S_{2}$ on $M_{1}$, $M_{2}$, respectively. It turns out that the validity of $\lambda_{0}(S_{2}) = \lambda_{0}(S_{1})$ is intertwined with a property similar, but weaker than the amenability of the covering.

The covering $p$ is called relatively amenable with respect to $q$, or for short, $q$-amenable if the monodromy action of $q_{*}\pi_{1}(M_{1})$  on the fiber of $q \circ p$ is amenable. It is evident that a covering is amenable if and only if it is relatively amenable with respect to the identity. 
The notion of relative amenability is naturally related to the amenability of the composition. More specifically, the composition $q \circ p$ is amenable if and only if $q$ is amenable and $p$ is $q$-amenable.
In general, amenable coverings are relatively amenable, but there exist relatively amenable coverings that are not amenable, as we will show by an example. It is worth to point out that if $q$ is finite sheeted or more importantly, if $q \circ p$ is normal, then $p$ is $q$-amenable if and only if it is amenable.
Our first result illustrates the role of this notion in the behavior of the bottom of the spectrum, extending \cite[Theorem 1.2]{BMP1} and \cite[Theorem 1.2]{Mine2}.

\begin{theorem}\label{main}
Let $p \colon M_{2} \to M_{1}$ and $q \colon M_{1} \to M_{0}$ be Riemannian coverings, with $q$ normal. Consider a Schr\"{o}dinger operator $S_{0}$ on $M_{0}$ and denote by $S_{1}$, $S_{2}$ its lift on $M_{1}$, $M_{2}$, respectively. Then:
\begin{enumerate}[topsep=0pt,itemsep=-1pt,partopsep=1ex,parsep=0.5ex,leftmargin=*, label=(\roman*), align=left, labelsep=0em]
\item If $p$ is $q$-amenable, then $\lambda_{0}(S_{2}) = \lambda_{0}(S_{1})$.

\item Conversely, if $\lambda_{0}(S_{2}) = \lambda_{0}(S_{1}) < \lambda_{0}^{\ess}(S_{0})$, then $p$ is $q$-amenable.
\end{enumerate}
\end{theorem}

It should be noticed that if $q$ is infinite sheeted, then $\lambda_{0}(S_{1}) = \lambda_{0}^{\ess}(S_{1})$ (cf. for example \cite[Corollary 1.3]{Mine}). Therefore, the setting of Theorem \ref{main} fits into the context of \cite[Question 1.5]{surv}. Conceptually, it seems interesting that even in this special case, a property different from amenability has such an effect on the behavior of the bottom of the spectrum.

In view of Theorem \ref{main}, the study of spectral-tightness becomes quite easier, bearing in mind that in this theorem, if $M_{2}$ is simply connected, then $p$ is $q$-amenable if and only if it is amenable. More precisely, it follows that spectral-tightness of a closed manifold is a topological property determined by its fundamental group as follows:

\begin{theorem}\label{tightness}
A closed Riemannian manifold is spectrally-tight if and only if the unique normal, amenable subgroup of its fundamental group is the trivial one.
\end{theorem}

Finally, we focus on non-positively curved, closed Riemannian manifolds. Consider such a manifold $M$. According to the de Rham decomposition theorem, the universal covering space of $M$ is written as a Riemannian product $H_{0} \times \dots \times H_{k}$ for some $k \in \mathbb{N}$, where $H_{0}$ is a Euclidean space possibly of dimension zero, and $H_{i}$ is irreducible (that is, $H_{i}$ cannot be written as a Riemannian product of two manifolds of positive dimension) for $1 \leq i \leq k$. Moreover, this decomposition is unique up to order and isometric equivalence. The space $H_{0}$ is called the Euclidean local de Rham factor of $M$. According to \cite[Theorem]{Eberlein}, the dimension of $H_{0}$ coincides with the rank of the unique maximal normal abelian subgroup of $\pi_{1}(M)$.
Using this result and Theorem \ref{tightness}, we characterize the property of spectral-tightness for non-positively curved, closed Riemanniam manifolds.

\begin{theorem}\label{neg curv}
A non-positively curved, closed Riemannian manifold is spectrally-tight if and only if the dimension of its Euclidean local de Rham factor is zero.
\end{theorem}

The notion of spectral-tightness seems more complicated on non-compact Riemannian manifolds. However, it is worth to mention that we prove the characterization of Theorem \ref{tightness} for any Riemannian manifold $M$ with $\lambda_0^{\ess}(M) > \lambda_0(\tilde{M})$, where $\tilde{M}$ is the universal covering space of $M$. Exploiting this, we present examples demonstrating that spectral-tightness of a non-compact Riemannian manifold depends on its Riemannian metric.

As another application of this characterization, we establish the spectral-tightness of certain geometrically finite manifolds, in the sense of Bowditch \cite{Bowditch}. Let $M$ be a complete Riemannian manifold with bounded sectional curvature $-b^2 \leq K \leq - a^2 < 0$. Then its universal covering space $\tilde{M}$ is a Hadamard manifold, and $M$ is written as a quotient $\tilde{M}/\Gamma$. Denoting by $\tilde{M}_c = \tilde{M} \cup \tilde{M}_{i}$ the geometric compactification of $\tilde{M}$, let $\Lambda \subset \tilde{M}_i$ be the limit set of $\Gamma$. Then $\Gamma$ acts properly discontinuously on $\tilde{M}_{c} \smallsetminus \Lambda$. The manifold $M$ is called geometrically finite if the topological manifold $(\tilde{M}_c \smallsetminus \Lambda) / \Gamma$ with possibly empty boundary has finitely many ends and each of them is parabolic. It is worth to mention that if $M$ is geometrically finite, then $\lambda_0^{\ess}(M) \geq \lambda_0(\tilde{M})$, according to \cite[Theorem B]{BP}.

\begin{corollary}\label{geom finite}
Let $M$ be a non-simply connected, geometrically finite manifold, and denote by $\tilde{M}$ its universal covering space. If $\lambda_0^{\ess}(M) > \lambda_0(\tilde{M})$, then $M$ is not spectrally-tight if and only if $\pi_1(M)$ is elementary.
\end{corollary}

The paper is organized as follows: In Section \ref{Preliminaries}, we give some preliminaries on Schr\"{o}dinger operators, amenable actions and coverings. In Section \ref{Rel amen}, we introduce the notion of relatively amenable coverings and discuss some basic properties. Section \ref{Spectrum} is devoted to Theorem \ref{main} and some applications. In Section \ref{Tightness}, we focus on the notion of spectral-tightness and establish Theorems \ref{tightness}, \ref{neg curv}, and Corollary \ref{geom finite}. \medskip

\textbf{Acknowledgements.} I would like to thank Werner Ballmann for some very helpful comments and remarks. I am also grateful to the Max Planck Institute for Mathematics in Bonn for its support and hospitality.

\section{Preliminaries}\label{Preliminaries}

Throughout this paper, manifolds are assumed to be connected and without boundary, unless otherwise stated. Moreover, non-connected manifolds are assumed to have at most countably many connected components.

A \textit{Schr\"{o}dinger operator} $S$ on a possibly non-connected Riemannian manifold $M$ is an operator of the form $S = \Delta + V$, where $\Delta$ is the Laplacian and $V \in C^{\infty}(M)$, such that there exists $c \in \mathbb{R}$ satisfying
\[
\langle Sf , f \rangle_{L^{2}(M)} \geq c \| f \|^{2}_{L^{2}(M)}
\]
for any $f \in C^{\infty}_{c}(M)$. Then the linear operator
\[
S \colon C^{\infty}_{c}(M) \subset L^{2}(M) \to L^{2}(M)
\]
is densely defined, symmetric and bounded from below. Hence, it admits Friedrichs extension. The (essential) spectrum of the Friedrichs extension of this operator is called the (essential, respectively) spectrum of $S$.

The spectrum and the essential spectrum of $S$ are denoted by $\sigma(S)$ and $\sigma_{\ess}(S)$, and their bottoms (that is, their infimums) by $\lambda_{0}(S)$ and $\lambda_{0}^{\ess}(S)$, respectively. In the case of the Laplacian (that is, $V=0$), these sets and quantities are denoted by $\sigma(M)$, $\sigma_{\ess}(M)$ and $\lambda_{0}(M)$, $\lambda_{0}^{\ess}(M)$, respectively. We have by definition that $\lambda_{0}^{\ess}(S) = + \infty$ if $\sigma_{\ess}(S)$ is empty, and we then say that $S$ has \textit{discrete spectrum}.

The \textit{Rayleigh quotient} of a non-zero $f \in \Lip_{c}(M)$ with respect to $S$ is defined by
\begin{equation}\label{Def Rayleigh quotient}
\mathcal{R}_{S}(f) = \frac{\int_{M} (\| \grad f \|^{2} + V f^{2})}{\int_{M} f^{2}}.
\end{equation}
The Rayleigh quotient of $f$ with respect to the Laplacian is denoted by $\mathcal{R}(f)$. According to the following proposition, the bottom of the spectrum of $S$ is expressed as an infimum of Rayleigh quotients.

\begin{proposition}\label{bottom}
Let $S$ be a Schr\"{o}dinger operator on a possibly non-connected Riemannian manifold $M$. Then the bottom of the spectrum of $S$ is given by
\[
\lambda_{0}(S) = \inf_{f} \mathcal{R}_{S}(f),
\]
where the infimum is taken over all $f \in C^{\infty}_{c}(M) \smallsetminus \{0\}$, or over all $f \in \Lip_{c}(M) \smallsetminus \{0\}$.
\end{proposition}

In the case where $M$ is connected, the bottom of the spectrum of $S$ is characterized as the maximum of the positive spectrum of $S$ (cf. for instance \cite[Theorem 3.1]{surv} and the references therein).

\begin{proposition}\label{positive spectrum}
Let $S$ be a Schr\"{o}dinger operator on a Riemannian manifold $M$. Then $\lambda_{0}(S)$ is the maximum of all $\lambda \in \mathbb{R}$ such that there exists a positive $\varphi \in C^{\infty}(M)$ satisfying $S \varphi = \lambda\varphi$.
\end{proposition}

It should be emphasized that the positive, smooth functions involved in this proposition are not required to be square-integrable.

We now focus on the essential spectrum of a Schr\"{o}dinger operator $S$ on a (connected) Riemannian manifold $M$. The decomposition principle asserts that
\[
\sigma_{\ess}(S) = \sigma_{\ess}(S , M \smallsetminus K)
\]
for any smoothly bounded, compact domain $K$ of $M$. This is well known in the case where $M$ is complete (compare with \cite[Proposition 2.1]{MR544241}), but also holds if $M$ is non-complete, as explained for example in \cite[Theorem A.17]{surv}. This yields the following expression for the bottom of the essential spectrum of $S$.

\begin{proposition}\label{bottom essential}
Let $S$ be a Schr\"{o}dinger operator on a Riemannian manifold $M$, and $(K_{n})_{n \in \mathbb{N}}$ an exhausting sequence of $M$ consisting of compact domains of $M$. Then the bottom of the essential spectrum of $S$ is given by
\[
\lambda_{0}^{\ess}(S) = \lim_{n} \lambda_{0}(S, M \smallsetminus K_{n}).
\]
\end{proposition}

Consider a positive $\varphi \in C^{\infty}(M)$ with $S \varphi = \lambda \varphi$ for some $\lambda \in \mathbb{R}$. Denote by $L^{2}_{\varphi}(M)$ the $L^{2}$-space of $M$ with respect to the measure $\varphi^{2} \dv$, where $\dv$ stands for the volume element of $M$ induced from its Riemannian metric. It is immediate to verify that the map $m_{\varphi} \colon L^{2}_{\varphi}(M) \to L^{2}(M)$ defined by $m_{\varphi}(f) = f \varphi$, is an isometric isomorphism. It is easily checked that $m_{\varphi}$ intertwines $S - \lambda$ with the diffusion operator
\[
L = m_{\varphi}^{-1} \circ (S - \lambda) \circ m_{\varphi} = \Delta - 2 \grad \ln \varphi.
\]
The operator $L$ is called the \textit{renormalization} of $S$ with respect to $\varphi$.
The \textit{Rayleigh quotient} of a non-zero $f \in C^{\infty}_{c}(M)$ with respect to $L$ is defined as
\[
\mathcal{R}_{L}(f) = \frac{\langle Lf , f \rangle_{L^{2}_{\varphi}(M)}}{\| f \|^{2}_{L^{2}_{\varphi}(M)}} = \frac{\int_{M} \| \grad f \|^{2} \varphi^{2}}{\int_{M} f^{2} \varphi^{2}}.
\]

\begin{proposition}\label{renormalization}
The Rayleigh quotients of any non-zero $f \in C^{\infty}_{c}(M)$ are related by $\mathcal{R}_{L}(f) = \mathcal{R}_{S}(f \varphi) - \lambda$. In particular, we have that
\[
\lambda_{0}(S) - \lambda = \inf_{f} \mathcal{R}_{L}(f),
\]
where the infimum is taken over all non-zero $f \in C^{\infty}_{c}(M)$.
\end{proposition}

\begin{proof}
The first equality follows from a straightforward computation, using the definition of $L$ and that $m_{\varphi}$ is an isometric isomorphism. This, together with Proposition \ref{bottom}, implies the second statement.\qed
\end{proof}\medskip

Even though our main results involve manifolds without boundary, it is quite important to consider manifolds with boundary in intermediate steps. Let $M$ be a possibly non-connected Riemannian manifold with smooth boundary, and denote by $\nu$ the outward pointing, unit normal to the boundary. Then the Laplacian on $M$ regarded as
\[
\Delta \colon \{ f \in C^{\infty}_{c}(M) : \nu(f) = 0 \text{ on } \partial M \} \subset L^{2}(M) \to L^{2}(M)
\]
admits Friedrichs extension, being densely defined, symmetric and bounded from below. The spectrum of the Friedrichs extension of this operator is called the \textit{Neumann spectrum} of $M$, and its bottom is denoted by $\lambda_{0}^{N}(M)$. We recall the following expression for the bottom of the Neumann spectrum, where $\mathcal{R}(f)$ is defined as in (\ref{Def Rayleigh quotient}) with $V=0$. This may be found for instance in \cite[Proposition 3.2]{Mine2}.

\begin{proposition}\label{bottom Neumann}
Let $M$ be a possibly non-connected Riemannian manifold with smooth boundary. Then the bottom of the Neumann spectrum of $M$ is given by
\[
\lambda_{0}^{N}(M) = \inf_{f} \mathcal{R}(f),
\]
where the infimum is taken over all non-zero $f \in C^{\infty}_{c}(M)$.
\end{proposition}

It should be noticed that in this proposition, the test functions $f \in C^{\infty}_{c}(M)$ do not have to satisfy any boundary condition.

\subsection{Amenable actions and coverings}

Let $X$ be a countable set and consider a right action of a discrete, countable group $\Gamma$ on $X$. This action is called \textit{amenable} if there exists an \textit{invariant mean} on $\ell^{\infty}(X)$; that is, a linear functional $\mu \colon \ell^{\infty}(X) \to \mathbb{R}$ such that
\[
\inf f \leq \mu(f) \leq \sup f \text{ and } \mu(g^{*} f) = \mu(f)
\]
for any $f \in \ell^{\infty}(X)$ and $g \in \Gamma$, where $g^{*}f(x) := f(xg)$ for any $x \in X$. It should be observed that if the action of $\Gamma$ on the orbit of some $x \in X$ is amenable, then the action of $\Gamma$ on $X$ is amenable.

A group $\Gamma$ is called \textit{amenable} if the right action of $\Gamma$ on itself is amenable. Standard examples of amenable groups are solvable groups and finitely generated groups of subexponential growth. It is worth to mention that the free group in two generators, as well as any group containing it, is non-amenable. It is not difficult to see that if $\Gamma$ is an amenable group, then any action of $\Gamma$ is amenable.

The following characterization of amenability is due to F\o{}lner.

\begin{proposition}
The right action of $\Gamma$ on $X$ is amenable if and only if for any $\varepsilon > 0$ and any finite subset $G$ of $\Gamma$ there exists a finite subset $F$ of $X$ such that $|F g \smallsetminus F| < \varepsilon |F|$ for any $g \in G$.
\end{proposition}

In particular, it follows that the right action of $\Gamma$ on $X$ is amenable if and only if the right action of any finitely generated subgroup of $\Gamma$ on $X$ is amenable. Moreover, if the action of $\Gamma$ on $X$ has finitely many orbits $X_{i}$, $1 \leq i \leq n$, then the action of $\Gamma$ on $X$ is amenable if and only if the action of $\Gamma$ on $X_{i}$ is amenable for some $1 \leq i \leq n$.

Let $p \colon M_{2} \to M_{1}$ be a smooth covering, where $M_{1}$ has possibly empty, smooth boundary and $M_{2}$ is possibly non-connected. Fix a point $x \in M_{1}$ and consider the fundamental group $\pi_{1}(M_{1})$ with base point $x$. For $g \in \pi_{1}(M_{1})$, let $\gamma_{g}$ be a representative loop of $g$ based at $x$. Given $y \in p^{-1}(x)$, let $\tilde{\gamma}_{g}$ be the lift of $\gamma_{g}$ starting at $y$ and denote its endpoint by $yg$. In this way, we obtain a right action of $\pi_{1}(M_{1})$ on $p^{-1}(x)$, which is called the \textit{monodromy action} of the covering. The covering $p$ is called \textit{amenable} if its monodromy action is amenable. It is easy to see that if $M_{2}$ is connected and $p$ is normal, then $p$ is amenable if and only if its deck transformation group is amenable.

\begin{example}
For any smooth covering $p \colon M_{2} \to M_{1}$, the covering
\[
p \sqcup \text{Id} \colon M_{2} \sqcup M_{1} \to M_{1}
\]
is amenable.
\end{example}

Recall that  F\o{}lner's condition characterizes the amenability of an action in terms of the action of finitely generated subgroups. In the context of coverings, this yields the following characterization in terms of smoothly bounded, compact domains.

\begin{proposition}[{\cite[Proposition 2.14]{surv}}]\label{amen domains}
Let $p \colon M_{2} \to M_{1}$ be a smooth covering, where $M_{2}$ is possibly non-connected, and $(K_{n})_{n \in \mathbb{N}}$ an exhausting sequence of $M_{1}$ consisting of smoothly bounded, compact domains. Then $p$ is amenable if and only the restriction $p \colon p^{-1}(K_{n}) \to K_{n}$ is amenable for any $n \in \mathbb{N}$.
\end{proposition}

This demonstrates the importance of considering non-connected covering spaces, since $p^{-1}(K)$ does not have to be connected even if $M_{2}$ is connected.

Finally, we briefly recall some results on the bottom of the spectrum under Riemannian coverings, which will be used in the sequel. Let $p \colon M_{2} \to M_{1}$ be a Riemannian covering, with $M_{2}$ possibly non-connected, $S_{1}$ a Schr\"{o}dinger operator on $M_{1}$ and $S_{2}$ its lift on $M_{2}$. From Proposition \ref{positive spectrum} it is not hard to see that the bottoms of the spectra satisfy
\begin{equation}\label{general inequality}
\lambda_{0}(S_{1}) \leq \lambda_{0}(S_{2}).
\end{equation}
The validity of the equality is closely related to the amenability of the covering.

\begin{theorem}\label{amen thm}
Let $p \colon M_{2} \to M_{1}$ be an amenable Riemannian covering, with $M_{2}$ possibly non-connected. Consider a Schr\"{o}dinger operator $S_{1}$ on $M_{1}$ and its lift $S_{2}$ on $M_{2}$. Then $\lambda_{0}(S_{2}) = \lambda_{0}(S_{1})$.
\end{theorem}

In the case where $M_{2}$ is connected, this coincides with \cite[Theorem 1.2]{BMP1}. With very slight modifications, its proof extends \cite[Theorem 1.2]{BMP1} to the case where $M_{2}$ is possibly non-connected. This may be found also in \cite[Theorem A]{orbi}, which involves Riemannian coverings of orbifolds, where the covering space may be non-connected.

Amenability of a Riemannian covering of a compact manifold with smooth boundary is characterized in terms of the Neumann spectrum, according to the following analogue of Brooks' result \cite{Brooks}.

\begin{theorem}\label{Neumann thm}
Let $p \colon M_{2} \to M_{1}$ be a Riemannian covering, where $M_{1}$ is compact with smooth boundary and $M_{2}$ is possibly non-connected. Then $p$ is amenable if and only if $\lambda_{0}^{N}(M_{2}) = 0$.
\end{theorem}

\begin{proof}
The converse implication is known by \cite[Theorem 4.1]{Mine2}. For the other direction, consider a Riemannian metric on $M_{1}$ such that its boundary has a neighborhood isometric to a cylinder $\partial M_{1} \times [0,\varepsilon)$, and endow $M_{2}$ with the lifted metric. Since this metric is uniformly equivalent to the original, it suffices to show that $\lambda_{0}^{N}(M_{2}) = 0$ with respect to this metric. Denote by $2M_{i}$ the Riemannian manifold obtained by gluing two copies of $M_{i}$ along their boundaries, $i=1,2$. Then $p \colon M_{2} \to M_{1}$ extends to a Riemannian covering $2p \colon 2M_{2} \to 2M_{1}$. Choose $x \in \partial M_{1} \subset 2M_{1}$ as base point for $\pi_{1}(2M_{1})$, and observe that any loop $c$ based at $x$ is written as $c = c_{2n} \star \dots \star c_1$ for some paths (not necessarily loops, since $\partial M_1$ may be non-connected) $c_{i}$, with the image of $c_{2i-1}$ contained in $M_{1}$, and the image of $c_{2i}$ in $2M_{1} \smallsetminus M_{1}^{\circ}$, $1 \leq i \leq n$. Denote by $c_{i}^{\prime}$ the reflection of $c_{i}$ along $\partial M_{1}$, and observe that the lifts of $c_i$ and $c_i^\prime$ starting from the same point also have the same endpoint. It is now apparent that given $y \in (2p)^{-1}(x) = p^{-1}(x)$, the lifts of $c$ and $c_{2n}^\prime \star c_{2n-1} \star \dots \star c_{2}^\prime \star c_1$ starting at $y$ have the same endpoint. Since the image of the latter loop is contained in $M_1$, it is now easy to verify that $2p$ is amenable, $p$ being amenable. Since $2M_{1}$ is closed, we derive from Theorem \ref{amen thm} that $\lambda_{0}(2M_{2}) = 0$. In view of Proposition \ref{bottom}, this means that for any $\varepsilon > 0$ there exists $f \in C^{\infty}_{c}(2M_2) \smallsetminus \{0\}$ with $\mathcal{R}(f) < \varepsilon$. Without loss of generality, we may assume that $f$ is not identically zero neither on $M_2$ nor on $2M_2 \smallsetminus M_2$. Indeed, otherwise one may extend $f$ beyond $\partial M_2$ to obtain a function invariant under reflection along $\partial M_2$, with the same Rayleigh quotient and the aforementioned property. Then we readily see from Proposition \ref{bottom Neumann} that
\[
\varepsilon > \mathcal{R}(f) \geq \min\{ \mathcal{R}(f|_{M_2}) , \mathcal{R}(f|_{2M_2 \smallsetminus M_2^\circ}) \} \geq \min\{ \lambda_0^N(M_{2}) , \lambda_0^{N}(2M_2 \smallsetminus M_2^\circ) \}.
\]
Bearing in mind that $M_2$ and $2M_2 \smallsetminus M_2^\circ$ are isometric, we conclude that $\lambda_{0}^{N}(M_{2}) = 0$. \qed
\end{proof}\medskip

By virtue of the preceding theorem, we may reformulate Proposition \ref{amen domains} as follows.

\begin{corollary}\label{amen exh seq}
Let $p \colon M_{2} \to M_{1}$ be Riemannian covering, with $M_{2}$ possibly non-connected, and $(K_{n})_{n \in \mathbb{N}}$ an exhausting sequence of $M_{1}$ consisting of smoothly bounded, compact domains. Then $p$ is amenable if and only if $\lambda_{0}^{N}(p^{-1}(K_{n})) = 0$ for any $n \in \mathbb{N}$.
\end{corollary}

\section{Relatively amenable coverings}\label{Rel amen}

In this section, we introduce the notion of relatively amenable coverings and present some of their properties. Let $p \colon M_{2} \to M_{1}$ and $q \colon M_{1} \to M_{0}$ be smooth coverings, with $q$ normal. Fix a base point $x \in M_{1}$ and set $x_{0} = q(x)$. The covering $p$ is called \textit{relatively amenable with respect to} $q$, or for short, $q$-\textit{amenable} if the monodromy action of $q_{*}\pi_{1}(M_{1})$ on $(q \circ p)^{-1}(x_{0})$ is amenable.

To provide another description of this action, denote by $\Gamma$ the deck transformation group of $q$. Let $s \in \pi_{1}(M_{1})$ and  $\gamma_{s}$ a representative loop of $s$ based at $x$.
It is clear that for any $z \in (q \circ p)^{-1}(x_0)$, there exists a unique $g \in \Gamma$, such that $z \in p^{-1}(gx)$.
Then $xs$ is the endpoint of the lift of $g \circ \gamma_{s}$ starting at $z$.

For $g \in \Gamma$, consider the covering $p_{g} \colon M_{2} \to M_{1}$ defined by $p_{g} = g^{-1} \circ p$, and denote by $\hat{p} \colon \hat{M} \to M_{1}$ the induced covering
\[
\sqcup_{g \in \Gamma} p_{g} \colon \sqcup_{g \in \Gamma} M_{2} \to M_{1}.
\]
From the above discussion, we arrive at the following characterization of relatively amenable coverings.
\begin{lemma}\label{induced covering}
The covering $p$ is $q$-amenable if and only if the induced covering $\hat{p}$ is amenable.
\end{lemma}

\begin{proof}
It is easily checked that the monodromy action of $\hat{p}$ coincides with the monodromy action of $q_{*}\pi_{1}(M_{1})$ on $(q \circ p)^{-1}(x_{0})$. \qed
\end{proof}\medskip

In view of the preceding lemma, it is evident that amenable coverings are relatively amenable. Furthermore, we readily see that if $q$ is finite sheeted, then $p$ is $q$-amenable if and only if it is amenable. We now discuss an analogue of Corollary \ref{amen exh seq}.

\begin{proposition}\label{relative amenability domains}
Let $p \colon M_{2} \to M_{1}$ and $q \colon M_{1} \to M_{0}$ be Riemannian coverings, where $q$ is normal with deck transformation group $\Gamma$, and fix an exhausting sequence $(K_{n})_{n \in \mathbb{N}}$ of $M_{1}$ consisting of smoothly bounded, compact domains. Then $p$ is $q$-amenable if and only if
\[
\inf_{g \in \Gamma}\lambda_{0}^{N}(p^{-1}(g K_{n})) = 0
\]
for any $n \in \mathbb{N}$.
\end{proposition}

\begin{proof}
It is obvious that $\hat{p}^{-1}(K)$ is isometric to the disjoint union of $p^{-1}(gK)$ with $g \in \Gamma$, for any smoothly bounded, compact domain $K$ of $M_{1}$. In particular, the bottoms of the Neumann spectra are related by
\[
\lambda_{0}^{N}(\hat{p}^{-1}(K)) = \inf_{g \in \Gamma} \lambda_{0}^{N}(p^{-1}(gK)).
\]
The proof is completed by Corollary \ref{amen exh seq} and Lemma \ref{induced covering}. \qed
\end{proof}

\begin{corollary}\label{normal composition}
Let $p \colon M_{2} \to M_{1}$ and $q \colon M_{1} \to M_{0}$ be smooth coverings, with $q$ and $q \circ p$ normal. Then $p$ is $q$-amenable if and only if it is amenable.
\end{corollary}

\begin{proof}
It is clear that if $p$ is amenable, then it is $q$-amenable. For the converse implication, endow $M_{0}$ with a Riemannian metric and $M_{1}$, $M_{2}$ with the lifted metrics. Denote by $\Gamma$ the deck transformation group of $q$ and let $(K_{n})_{n \in \mathbb{N}}$ be an exhausting sequence of $M_{1}$ consisting of smoothly bounded, compact domains. Then Proposition \ref{relative amenability domains} states that
\[
\inf_{g \in \Gamma}\lambda_{0}^{N}(p^{-1}(g K_{n}))  = 0
\]
for any $n \in \mathbb{N}$. Since $q \circ p$ is normal, we deduce that any $g \in \Gamma$ can be lifted to an isometry of $M_{2}$. Indeed, given $x \in M_2$, write $x_0 = (q \circ p)(x)$ and observe that $(q_* \circ g_* \circ p_*)(\pi_1(M_2,x)) = (q_* \circ p_*)(\pi_1(M_2,x))$. Hence, obtain that $(q_* \circ g_* \circ p_*)(\pi_1(M_2,x)) = (q_* \circ p_*)(\pi_1 (M_2,y))$ for any $y \in p^{-1}(gp(x))$, $q \circ p$ being normal. Keeping in mind that $q_* \colon \pi_{1}(M_1,gp(x)) \to \pi_1(M_0,x_0)$ is injective, this implies that $(g_* \circ p_*)(\pi_1(M_2,x)) = p_*(\pi_1 (M_2,y))$. It now follows from the lifting theorem that $g$ can be lifted to a local isometry of $M_2$ mapping $x$ to $y$. Since $x \in M_2$, $g \in \Gamma$ and $y \in p^{-1}(gp(x))$ are arbitrary, we readily see that the lift is invertible, and hence, an isometry. In particular, we derive that  $p^{-1}(g K_{n})$ is isometric to $p^{-1}(K_{n})$ for any $g \in \Gamma$ and $n \in \mathbb{N}$, which yields that
\[
\lambda_{0}^{N}(p^{-1}(K_{n})) = \inf_{g \in \Gamma}\lambda_{0}^{N}(p^{-1}(g K_{n}))  = 0
\]
for any $n \in \mathbb{N}$. We conclude from Corollary \ref{amen exh seq} that $p$ is amenable. \qed
\end{proof}

\section{Spectrum under relatively amenable coverings}\label{Spectrum}

In this section we study the behavior of the bottom of the spectrum under relatively amenable coverings, and give some applications and examples. \medskip

\noindent\emph{Proof of Theorem \ref{main}:} 
Suppose first that $p$ is $q$-amenable. Then the induced covering $\hat{p}$ is amenable, from Lemma \ref{induced covering}. Denoting by $\hat{S}$ the lift of $S_{1}$ on $\hat{M}$, we derive from Theorem \ref{amen thm} that $\lambda_0(\hat{S}) = \lambda_0(S_1)$. Moreover, bearing in mind that $S_{1}$ is the lift of $S_{0}$, we readily see that $S_1$ is invariant under deck transformations of $q$. Therefore, any connected component of $\hat{M}$ is isometric to $M_{2}$ via an isometry that identifies $\hat{S}$ with $S_{2}$. This yields that $\lambda_0(\hat{S}) = \lambda_0(S_2)$, which establishes the asserted equality.

To prove the second assertion, notice that the assumption that $\lambda_{0}^{\ess}(S_{0}) > \lambda_{0}(S_{1})$, together with Proposition \ref{bottom essential}, implies that there exists a compact domain $D_{0}$ of $M_{0}$ such that
\[
\lambda_{0}(S_{0}, M_{0} \smallsetminus D_{0}) > \lambda_{0}(S_{1}).
\]
Let $(K_{m})_{m \in \mathbb{N}}$ be an exhausting sequence of $M_{1}$ consisting of smoothly bounded, compact domains, such that $D_{0}$ is contained in the interior of $q(K_{1})$. 

Assume to the contrary that $p$ is not $q$-amenable, and denote by $\Gamma$ the deck transformation group of $q$. By virtue of Proposition \ref{relative amenability domains}, there exists $m \in \mathbb{N}$ and $c > 0$, such that
\[
\lambda_{0}^{N}(p^{-1}(g K_{m})) \geq c
\]
for any $g \in \Gamma$. Set $K = K_{m}$ and $D = q(K_{m})$.

We know from Proposition \ref{positive spectrum} that there exists a positive $\varphi_{1} \in C^{\infty}(M_{1})$ satisfying $S_{1}\varphi_{1} = \lambda_{0}(S_{1}) \varphi_{1}$. Denote by $\varphi_{2}$ the lift of $\varphi_{1}$ on $M_{2}$ and by $L$ the renormalization of $S_{2}$ with respect to $\varphi_{2}$. Since $\lambda_{0}(S_{2}) = \lambda_{0}(S_{1})$, we derive from Proposition \ref{renormalization} that there exists $(f_{n})_{n \in \mathbb{N}} \subset C^{\infty}_{c}(M_{2})$ with $\| f_{n} \|_{L^{2}_{\varphi_{2}}(M_{2})} = 1$ and $\mathcal{R}_{L}(f_{n}) \rightarrow 0$.

According to the gradient estimate \cite[Theorem 6]{Cheng-Yau}, there exists $C>0$ such that
\[
\max_{K} \psi \leq C \min_{K} \psi
\]
for any positive $\psi \in C^{\infty}(M_{1})$ with $S_{1} \psi = \lambda_{0}(S_{1}) \psi$. For $\psi = \varphi_{1} \circ g$, this means that
\[
\max_{p^{-1}(gK)} \varphi_{2} = \max_{gK} \varphi_{1} \leq C \min_{gK} \varphi_{1} = C \min_{p^{-1}(gK)} \varphi_{2}
\]
for any $g \in \Gamma$. Using this and Proposition \ref{bottom Neumann}, we compute
\begin{eqnarray}\label{harnack}
\int_{p^{-1}(gK)} \| \grad f_{n} \|^{2} \varphi_{2}^{2} &\geq& \big( \min_{p^{-1}(gK)} \varphi_{2}^{2} \big) \int_{p^{-1}(gK)} \| \grad f_{n} \|^{2} \nonumber \\ 
&\geq& \frac{1}{C^{2}} \big( \max_{p^{-1}(gK)} \varphi_{2}^{2} \big) \lambda_{0}^{N}(p^{-1}(gK)) \int_{p^{-1}(gK)} f_{n}^{2} \nonumber \\
&\geq& \frac{c}{C^{2}} \int_{p^{-1}(gK)} f_{n}^{2} \varphi_{2}^{2}
\end{eqnarray}
for any $n \in \mathbb{N}$ and $g \in \Gamma$.

Since $K$ is compact, it is clear that there exists $k \in \mathbb{N}$ such that any point of $M_{1}$ belongs to at most $k$ different translates $gK$ of $K$, with $g \in \Gamma$. Thus, any point of $M_{2}$ belongs to at most $k$ different $p^{-1}(gK)$, with $g \in \Gamma$. This, together with (\ref{harnack}), gives the estimate
\[
\int_{(q \circ p)^{-1}(D)} \| \grad f_{n} \|^{2} \varphi_{2}^{2} \geq \frac{1}{k} \sum_{g \in \Gamma} \int_{p^{-1}(gK)} \| \grad f_{n} \|^{2} \varphi_{2}^{2} \geq \frac{c}{kC^{2}} \int_{(q \circ p)^{-1}(D)} f_{n}^{2} \varphi_{2}^{2},
\]
where we used that $(q \circ p)^{-1}(D)$ is the union of $p^{-1}(gK)$ with $g \in \Gamma$.
Bearing in mind that $\| f_{n} \|_{L^{2}_{\varphi_{2}}(M_{2})} = 1$ and $\mathcal{R}_{L}(f_{n}) \rightarrow 0$, this yields that
\begin{equation}\label{convergence}
\int_{(q \circ p)^{-1}(D)} f_{n}^{2} \varphi_{2}^{2} \rightarrow 0 \text{ and } \int_{M_{2} \smallsetminus (q \circ p)^{-1}(D)} f_{n}^{2} \varphi_{2}^{2} \rightarrow 1.
\end{equation}

Let $\chi_{0} \in C^{\infty}_{c}(M_{0})$ with $\chi_{0} = 1$ in a neighborhood of $D_{0}$ and $\supp \chi_{0}$ contained in the interior of $D$. Denote by $\chi_{2}$ the lift of $\chi_{0}$ on $M_{2}$, and set $h_{n} = (1 - \chi_{2})f_{n} \in C^{\infty}_{c}(M_{2})$. In view of (\ref{convergence}), it is immediate to verify that
\[
\int_{M_{2}} h_{n}^{2}\varphi_{2}^{2} \rightarrow 1 \text{ and } \int_{M_{2}} \| \grad h_{n} \|^{2} \varphi_{2}^{2} \rightarrow 0,
\]
which shows that that $\mathcal{R}_{L}(h_{n}) \rightarrow 0$, and thus, $\mathcal{R}_{S_{2}}(h_{n} \varphi_{2}) \rightarrow \lambda_{0}(S_{1})$, by Proposition \ref{renormalization}. On the other hand, we have that $h_{n} \varphi_{2}$ is compactly supported in $M_{2} \smallsetminus (q \circ p)^{-1}(D_{0})$. By Proposition \ref{bottom}, this implies that
\[
\mathcal{R}_{S_{2}}(h_{n}\varphi_{2}) \geq \lambda_{0}(S_{2} , M_{2} \smallsetminus (q \circ p)^{-1}(D_{0})) \geq \lambda_{0}(S_{0} , M_{0} \smallsetminus D_{0}) > \lambda_{0}(S_{1}),
\]
where the intermediate inequality follows from (\ref{general inequality}) applied to the restriction of $q \circ p$ over any connected component of $M_{0} \smallsetminus D_{0}$. This is a contradiction, which completes the proof. \qed \medskip

The next example illustrates that in Theorem \ref{main}(ii), the covering $p$ is $q$-amenable, but not necessarily amenable, and in particular, that relative amenability is indeed a weaker property than amenability.
This example also demonstrates that in Corollary \ref{normal composition}, the assumption that $q \circ p$ is normal cannot be replaced with $p$ being normal.

\begin{example}
Let $N$ be a closed Riemannian manifold of dimension $n \geq 3$ with non-amenable fundamental group. Fix two sufficiently small open balls $B_{i}$ with disjoint closures and $\partial B_{i}$ diffeomorphic to the sphere $S^{n-1}$, $i=1,2$. Denote by $M_{0}$ the closed manifold obtained by gluing a cylinder $S^{n-1} \times [0,1]$ along the boundary of $N \smallsetminus (B_{1} \cup B_{2})$, so that $S^{n-1} \times \{0\}$ gets identified with $\partial B_{1}$, and $S^{n-1} \times \{1\}$ with $\partial B_{2}$. Endow $M_{0}$ with a Riemannian metric.

Consider now the disjoint union of copies $N_{k}$ of $N \smallsetminus (B_{1} \cup B_{2})$, with $k \in \mathbb{Z}$. For each $k \in \mathbb{Z}$, glue a cylinder $S^{n-1} \times [0,1]$ along the boundary of this disjoint union, so that $S^{n-1} \times \{0\}$ gets identified with $\partial B_{1}$ in $N_{k}$, and $S^{n-1} \times \{1\}$ with $\partial B_{2}$ in $N_{k+1}$. In this way, we obtain a manifold $M_{1}$ on which $\mathbb{Z}$ acts via diffeomorphsms and the quotient is diffeomorphic to $M_{0}$. That is, we have a covering $q \colon M_{1} \to M_{0}$ with deck transformation group $\mathbb{Z}$. We endow $M_{1}$ with the lift of the Riemannian metric of $M_{0}$.

Let $F$ be a fundamental domain of $q$ which is diffeomorphic to $N_{0}$ with two cylinders $S^{n-1} \times [0,1/2]$ attached along its boundary. Then $\pi_{1}(F)$ is non-amenable and $\partial F$ has two connected components $C_{1}$ and $C_{2}$, which are diffeomorphic to $S^{n-1}$. Hence, the universal covering $p \colon \tilde{F} \to F$ of $F$ is non-amenable. It should be noticed that the restriction of $p$ on any connected component of $\partial \tilde{F}$ is a covering over a connected component of $\partial F$. Since the connected components of $\partial F$ are diffeomorphic to $S^{n-1}$, where $n \geq 3$, we derive that $p$ restricted to any connected component of $\partial \tilde{F}$ is an isometry. This means that $p^{-1}(C_{i})$ is a disjoint union of copies of $C_{i}$, $i=1,2$. Write $M_{1} \smallsetminus F^{\circ} = D_{1} \sqcup D_{2}$, with $C_{i}$ contained in $D_{i}$, $i=1,2$. Denote by $M_{2}$ the manifold obtained by gluing a copy of $D_{i}$ to $\tilde{F}$ along any connected component of $p^{-1}(C_{i})$, $i=1,2$. We readily see that the covering $p \colon \tilde{F} \to F$ is extended to a normal covering $p \colon M_{2} \to M_{1}$ with the same deck transformation group. Therefore, $p \colon M_{2} \to M_{1}$ is non-amenable.

It remains to establish that $p \colon M_{2} \to M_{1}$ is $q$-amenable. To this end, we endow $M_{2}$ with the lift of the Riemannian metric of $M_{1}$, and by virtue of Theorem \ref{main}, it suffices to prove that $\lambda_{0}(M_{2}) = 0$. Since $q$ is amenable, Theorem \ref{amen thm} states that $\lambda_{0}(M_{1}) = 0$. From Proposition \ref{bottom}, for any $\varepsilon > 0$ there exists a non-zero $f \in C^{\infty}_{c}(M_{1})$ with $\mathcal{R}(f) < \varepsilon$. Then there exists a deck transformation $g$ of $q$, such that $\supp (f \circ g)$ is contained in the interior of $D_{1}$,  $\supp f$ being compact. It is evident that $\mathcal{R}(f \circ g) < \varepsilon$, and since $D_{1}$ is isometric to a domain of $M_{2}$, the corresponding $f_{2} \in C^{\infty}_{c}(M_{2})$ satisfies $\mathcal{R}(f_{2}) < \varepsilon$. Since $\varepsilon > 0$ is arbitrary, we conclude from Proposition \ref{bottom} that $\lambda_{0}(M_{2}) = 0$, as we wished.
\end{example}

Based on the preceding example, we now present examples of amenable smooth coverings $p \colon M_{2} \to M_{1}$ with $M_{2}$ non-connected, such that the restriction of $p$ on any connected component of $M_{2}$ is non-amenable.

\begin{example}
Let $p \colon  M_{2} \to M_{1}$ and $q \colon M_{1} \to M_{0}$ be smooth coverings, where $q$ is normal with deck transformation group $\Gamma$. Suppose that $p$ is $q$-amenable, but not amenable. Then the induced covering $\hat{p} \colon \hat{M} \to M_{1}$ is amenable, from Lemma \ref{induced covering}. However, the restriction of $\hat{p}$ on any connected component of $\hat{M}$ is of the form $p_{g} = g^{-1} \circ p \colon M_{2} \to M_{1}$ for some $g \in \Gamma$. It follows from Corollary \ref{amen exh seq} that $p_{g}$ is non-amenable for any $g \in \Gamma$, $p$ being non-amenable.
\end{example}

We end this section with a characterization of the amenability of a composition of coverings.

\begin{corollary}\label{normal}
Let $p \colon M_{2} \to M_{1}$ and $q \colon M_{1} \to M_{0}$ be smooth coverings, with $q$ normal. Then $q \circ p$ is amenable if and only if $q$ is amenable and $p$ is $q$-amenable.
\end{corollary}

\begin{proof}
Endow $M_{0}$ with a Riemannian metric and $M_{1}$, $M_{2}$ with the lifted metrics. Let $S_{0}$ be a Schr\"{o}dinger operator on $M_{0}$ with discrete spectrum (for instance, $S_{0} = \Delta + V$, where $V$ has compact sublevels), and denote by $S_{1}$, $S_{2}$ its lift on $M_{1}$, $M_{2}$, respectively.

If $q \circ p$ is amenable, then $\lambda_{0}(S_{0}) = \lambda_{0}(S_{2})$, from Theorem \ref{amen thm}. 
Furthermore, we know from (\ref{general inequality}) that $\lambda_{0}(S_{0}) \leq \lambda_{0}(S_{1}) \leq \lambda_{0}(S_{2})$. Hence, we have that $\lambda_{0}(S_{1}) = \lambda_{0}(S_{0}) < \lambda_{0}^{\ess}(S_{0})$, which yields that $q$ is amenable, from \cite[Theorem 1.2]{Mine2}. Since $\lambda_{0}(S_{2}) = \lambda_{0}(S_{1}) < \lambda_{0}^{\ess}(S_{0})$, Theorem \ref{main} implies that $p$ is $q$-amenable.

Conversely, if $q$ is amenable and $p$ is $q$-amenable, then $\lambda_{0}(S_{2}) = \lambda_{0}(S_{1}) = \lambda_{0}(S_{0})$, where we used Theorem \ref{main} for the first equality, and Theorem \ref{amen thm} for the second one. Since $S_{0}$ has discrete spectrum, \cite[Theorem 1.2]{Mine2} states that $q \circ p$ is amenable.\qed
\end{proof}

\section{Spectral-tightness of Riemannian manifolds}\label{Tightness}

In this section we study the notion of spectral-tightness of Riemannian manifolds. Recall that a Riemannian manifold $M$ is called \textit{spectrally-tight} if $\lambda_{0}(M^{\prime}) < \lambda_{0}(\tilde{M})$ for any non-simply connected, normal covering space $M^{\prime}$ of $M$, where $\tilde{M}$ is the universal covering space of $M$. We begin with a straightforward consequence of Theorem \ref{main}.

\begin{corollary}\label{tightness general}
Let $M$ be a Riemannian manifold with $\lambda_{0}(\tilde{M}) < \lambda_{0}^{\ess}(M)$, where $\tilde{M}$ is the universal covering space of $M$. Then $M$ is spectrally-tight if and only if the unique normal, amenable subgroup of $\pi_{1}(M)$ is the trivial one.
\end{corollary}

\begin{proof}
If $M$ is not spectrally-tight, then there exists a normal covering $q \colon M^{\prime} \to M$ with $\lambda_{0}(M^{\prime}) = \lambda_{0}(\tilde{M})$ and $M^{\prime}$ non-simply connected. In view of Theorem \ref{main}, the universal covering $p \colon \tilde{M} \to M^{\prime}$ is $q$-amenable. Since $q \circ p$ is normal, Corollary \ref{normal composition} asserts that $p$ is amenable, or equivalently, that $\pi_{1}(M^{\prime})$ is amenable. It is clear that $\pi_{1}(M^{\prime})$ is a non-trivial, normal subgroup of $\pi_{1}(M)$.

Conversely, if $\pi_{1}(M)$ has a non-trivial, amenable, normal subgroup $\Gamma$, then the action of $\Gamma$ on $\tilde{M}$ gives rise to an amenable Riemannian covering $p \colon \tilde{M} \to \tilde{M} / \Gamma$. We deduce from Theorem \ref{amen thm} that $\lambda_{0}(\tilde{M}) = \lambda_{0}(\tilde{M} / \Gamma)$, while $\tilde{M} / \Gamma$ is a non-simply connected, normal covering space of $M$. \qed
\end{proof}\medskip

\noindent\emph{Proof of Theorem \ref{tightness}:} It is an immediate consequence of Corollary \ref{tightness general}, since the spectrum of the Laplacian on a closed Riemannian manifold is discrete. \qed\medskip

\noindent\emph{Proof of Theorem \ref{neg curv}:} Let $M$ be a non-positively curved, closed Riemannian manifold. Suppose first that the dimension of the Euclidean local de Rham factor of $M$ is non-zero. We derive from \cite[Theorem]{Eberlein} that there exists a non-trivial, normal abelian subgroup of $\pi_{1}(M)$. By virtue of Theorem \ref{tightness}, this shows that $M$ is not spectrally-tight.

Conversely, if $M$ is not spectrally-tight, then there exists a non-trivial, normal amenable subgroup $\Gamma$ of $\pi_{1}(M)$, according to Theorem \ref{tightness}. We obtain from \cite[Corollary 2]{MR875344} that $\Gamma$ is a Bieberbach group. Denote by $G$ the subgroup of translations in $\Gamma$. It should be noticed that $G$ is non-trivial, since otherwise, $\Gamma$ is finite and in particular, there exist non-trivial elements of finite order in $\pi_{1}(M)$, which is a contradiction. Furthermore, it is known that $G$ is the unique maximal normal abelian subgroup of $\Gamma$, and thus, a characteristic subgroup of $\Gamma$. Since $\Gamma$ is a normal subgroup of $\pi_{1}(M)$, we readily see that so is $G$. Since $G$ is isomorphic to $\mathbb{Z}^n$ for some $n \in \mathbb{N}$, we conclude that the rank of the unique normal maximal abelian subgroup of $\pi_{1}(M)$ is non-zero and the proof is completed by \cite[Theorem]{Eberlein}. \qed \medskip

Before proceeding to the proof of Corollary \ref{geom finite}, we need the following observation, which will be also exploited in the examples in the sequel.

\begin{lemma}\label{fundamental group of geom finite}
Let $M$ be a geometrically finite manifold. If $\pi_1(M)$ contains a non-trivial, amenable normal subgroup, then $\pi_1(M)$ is elementary.
\end{lemma}

\begin{proof}
	Denote by $\tilde{M}$ the universal covering space of $M$ and write $M = \tilde{M}/\Gamma$.
	Let $G$ be a non-trivial, amenable normal subgroup of $\Gamma$. It follows from  \cite[Corollary 1.2]{MR2275833} that $G$ is elementary. Bearing in mind that finite groups of isometries of $\tilde{M}$ fix a point in $\tilde{M}$, this yields that $G$ is either parabolic or axial.
	
	Suppose first that $G$ is parabolic and contained in the stabilizer of a point $x$ of the ideal boundary of $\tilde{M}$. Given $h \in \Gamma$, using that $hgh^{-1} \in G$, we readily see that $G$ fixes $h^{-1}x$. This shows that $h^{-1}x = x$ for any $h \in \Gamma$, and thus, $\Gamma$ is parabolic and fixes $x$.
	
	Assume now that $G$ is axial and contained in the stabilizer of the image of a geodesic $\gamma \colon \R \to \tilde{M}$ joining two points of the ideal boundary of $\tilde{M}$. Then for any $g \in G$, there exists $t_g \in \R$ such that $g(\gamma(t)) = \gamma(t + t_g)$ for any $t \in \R$. Given $h \in \Gamma$, there exists $t^\prime \in \R$ such that $hgh^{-1}(\gamma(t)) = \gamma(t+t^\prime)$ for any $t \in \R$, since $hgh^{-1} \in G$. Then $G$ fixes the image of the geodesic $h^{-1} \circ \gamma$, which implies that the images of $h^{-1} \circ \gamma$ and $\gamma$ coincide for any $h \in \Gamma$. We conclude that $\Gamma$ is axial and fixes the image of $\gamma$.\qed
\end{proof}\medskip

\noindent\emph{Proof of Corollary \ref{geom finite}:}
If $M$ is not spectrally-tight, we derive from Corollary \ref{tightness general} that there exists a non-trivial, amenable normal subgroup of $\pi_1(M)$. Then $\pi_1(M)$ is elementary, by Lemma \ref{fundamental group of geom finite}. Conversely, if $\pi_1(M)$ is elementary, then $\lambda_0(M) = \lambda_0(\tilde{M})$, in view of Theorem \ref{amen thm}, $\pi_1(M)$ being amenable. Since $M$ is non-simply connected, this means that $M$ is not spectrally-tight.\qed\medskip

Even though spectral-tightness of closed Riemannian manifolds is a topological property, the next examples illustrate that for non-compact manifolds, this property depends on the Riemannian metric.

\begin{example}
	Let $M$ be a non-compact surface of finite type with non-cyclic fundamental group, and denote by $\tilde{M}$ its universal covering space. Since $M$ is diffeomorphic to a closed surface with finitely many points removed, it is not hard to see that $M$ carries a complete Riemannian metric which is flat outside a compact domain. We derive from Proposition \ref{bottom essential} that $\lambda_0^{\ess}(M) = 0$. Furthermore, there exists a compact $K \subset M$ such that the fundamental group of any connected component of $M \smallsetminus K$ is amenable (as a matter of fact, cyclic). We deduce from \cite[Corollary 1.6]{Mine} that $\lambda_0(\tilde{M}) = 0$, and thus, $M$ is not spectrally-tight with respect to this Riemannian metric.
	
	It follows from \cite[Proposition 1.5 and (1.3)]{BMM} that $M$ carries a complete Riemannian metric such that $\lambda_0^{\ess}(M) > \lambda_0(\tilde{M})$. We now show that $M$ is spectrally-tight with respect to this Riemannian metric. Otherwise, we obtain from Theorem \ref{tightness general} that there exists a non-trivial, amenable normal subgroup of $\pi_1(M)$. It is known that $M$ admits a complete hyperbolic metric, and with respect to such a metric, $M$ is geometrically finite (cf. for instance \cite[Example 1.4]{BMM}). It now follows from Lemma \ref{fundamental group of geom finite} that $\pi_1(M)$ is elementary, and thus, cyclic, $M$ being two-dimensional, which is a contradiction.
\end{example}

Although this gives a quite wide class of examples, it seems reasonable to present a more explicit one.

\begin{example}
Let $M$ be a two-dimensional torus with a cusp attached. Initially, we endow $M$ with a complete Riemannian metric, so that the cusp $D$ is isometric to a domain of a flat cylinder. Then it is easily checked that $\lambda_{0}(M) = \lambda_{0}(D) = \lambda_{0}^{\ess}(M) = 0$. Since the fundamental group of $D$ is amenable, it follows from \cite[Corollary 1.6]{Mine} that $\lambda_{0}(M) = \lambda_{0}(\tilde{M})$, where $\tilde{M}$ stands for the universal covering space of $M$. Therefore, $M$ endowed with this Riemannian metric is not spectrally-tight.

We now endow $M$ with a complete Riemannian metric, so that the cusp is isometric to the surface of revolution generated by $e^{-t^{\alpha}}$ with $t \geq 1$, for some $\alpha > 1$. Then the spectrum of $M$ is discrete (cf. \cite[Theorem 2]{Brooks2}) and Corollary \ref{tightness general} characterizes spectral-tightness. It should be observed that $\pi_{1}(M)$ is the free group $F_{2}$ in two generators. Since the fundamental group of any negatively curved, closed manifold contains a subgroup isomorphic to $F_{2}$, \cite[Theorem 1]{Chen} shows that any amenable subgroup of $F_{2}$ is cyclic. It is not difficult to verify that the unique normal, cyclic subgroup of $F_{2}$ is the trivial one. Hence, $M$ endowed with this Riemannian metric is spectrally-tight, from Corollary \ref{tightness general}.
\end{example}

\noindent Max Planck Institute for Mathematics \\
Vivatsgasse 7, 53111, Bonn \\
E-mail address: polymerp@mpim-bonn.mpg.de

\end{document}